\documentclass{amsart}

\usepackage{url}

\newtheorem{theorem}{Theorem}[section]
\newtheorem{proposition}[theorem]{Proposition}
\newtheorem{lemma}[theorem]{Lemma}
\newtheorem{corollary}[theorem]{Corollary}
\newtheorem{fact}[theorem]{Fact}
\newtheorem{assumption}[theorem]{Assumption}

\theoremstyle{definition}

\theoremstyle{remark}
\newtheorem{remark}[theorem]{Remark}
\newtheorem{example}[theorem]{Example}

\numberwithin{equation}{section}

\def\Ind{\setbox0=\hbox{$x$}\kern\wd0\hbox to 0pt{\hss$\mid$\hss} \lower.9\ht0\hbox to 0pt{\hss$\smile$\hss}\kern\wd0} 

\def\Notind{\setbox0=\hbox{$x$}\kern\wd0\hbox to 0pt{\mathchardef \nn=12854\hss$\nn$\kern1.4\wd0\hss}\hbox to 0pt{\hss$\mid$\hss}\lower.9\ht0 \hbox to 0pt{\hss$\smile$\hss}\kern\wd0}

\def \d {\operatorname{def}}

\def \Autd {\operatorname{Aut}_{\operatorname{def}}}
\def \Hd {H^1_{\d}(\G,\Autd(G))}
\def \Aut {\operatorname{Aut}}
\def \G {\mathcal G}
\def \P {\mathfrak P}

\title{More on Galois cohomology,  definability \\and differential algebraic groups}

\author{Omar Le\'on S\'anchez}
\address{Omar Le\'on S\'anchez\\
University of Manchester\\
Department of Mathematics\\
Oxford Road \\
Manchester, M13 9PL.}
\email{omar.sanchez@manchester.ac.uk}

\author{David Meretzky}
\address{David Meretzky\\
University of Notre Dame\\
Department of Mathematics\\
255 Hurley, Notre Dame\\
IN 46556.}
\email{dmeretzk@nd.edu}

\author{Anand Pillay}
\address{Anand Pillay\\
University of Notre Dame\\
Department of Mathematics\\
255 Hurley, Notre Dame\\
IN, 46556.}
\email{apillay@nd.edu}
\thanks{Anand Pillay was supported by NSF grants DMS-1665035 and  DMS-1760212}

\date{\today}

\begin{document}

\maketitle

\begin{abstract}
As a continuation of the work of the third author in \cite{Pillay-cohomology}, we make further observations on the features of Galois cohomology in the general model theoretic context. We make explicit the connection between forms of definable groups and first cohomology sets with coefficients in a suitable automorphism group. We then use a method of twisting cohomology (inspired on Serre's algebraic twisting) to describe arbitrary fibres in cohomology sequences -- yielding a useful ``finiteness'' result on cohomology sets. 

Applied to the special case of differential fields and Kolchin's constrained cohomology, we complete results from \cite{LS-Pillay} by proving that the first constrained cohomology set of a differential algebraic group over a bounded, differentially large, field is countable. 
\end{abstract}

\section{introduction}

As in the third author's paper \cite{Pillay-cohomology}, which develops some basic features of definable cohomology, in this note we are mainly concerned with translations from the Galois cohomological language to the model-theoretic language of definable sets and atomic extensions. Nonetheless, the general statements here do yield new results in the differential cohomology context; that is, when specialized to Kolchin's constrained cohomology, see Example~\ref{examples}(ii), Remark~\ref{finalremark}(iii), and \S\ref{DGC}.

\medskip

We work in the (general)  model-theoretic context of \cite[\S2]{Pillay-cohomology}; namely: 

\begin{assumption}\label{assumptions} We fix a (first-order) structure $M$ and an arbitrary subset $A$ such that 
\begin{enumerate}
\item [(i)] $M$ is atomic over $A$, meaning that for every (finite) tuple $a$ from $M$ the type $tp(a/A)$ is isolated,
\item [(ii)] for any (finite) tuples $a$ and $b$ from $M$, if $tp(a/A)=tp(b/A)$, then there is $\sigma\in \Aut(M/A)$ such that $\sigma(a)=b$, and
\item [(iii)] the theory of $M$ eliminates imaginaries (alternatively one could work in $M^{\operatorname{eq}}$).
\end{enumerate}
\end{assumption}

\noindent {\bf Notation and conventions.} Throughout we let $\G$ denote the automorphism group $\Aut(M/A)$. We also let $G$ be an arbitrary (not necessarily abelian) group definable in $M$ with parameters from $A$. Note that $\G$ acts naturally on $G$ by group automorphisms. We make $\G$ into a topological group  by taking 
$$\{\text{Fix}(a)\;: \; a \text{ is a finite tuple from } M\}$$
as open neighbourhoods at the identity. We further equip $G$ with its discrete topology, thus making the natural action of $\G$ on $G$ continuous. 

\bigskip

We briefly recall the basic notions (and some results) on definable cohomology from \cite{Pillay-cohomology}. By a (1-)cocycle of $\G$ with values in $G$ (with respect to the natural action of $\G$ on $G$) we mean a crossed-homomorphism from $\G$ to $G$; namely, a map 
$$\Phi:\G \to G$$ such that 
$$\Phi(\sigma_1\cdot \sigma_2)=\Phi(\sigma_1)\cdot \sigma_1(\Phi(\sigma_2)).$$
A cocycle $\Phi$ is said to be definable if there is a (finite) tuple $a$ from $M$ and an $A$-definable (partial) function $h(x,y)$ such that for any $\sigma\in \G$ we have
$$\Phi(\sigma)=h(a,\sigma(a)).$$
One readily checks that definable cocycles are continuous. The set of definable cocycles (from $\G$ to $G$) is denoted by $Z^{1}_{\d}(\G,G)$. Two cocycles $\Phi$ and $\Psi$ are said to be cohomologous if there is $b\in G$ such that 
$$\Psi(\sigma)=b^{-1}\cdot \Phi(\sigma)\cdot \sigma(b), \quad \text{ for all }\sigma\in \G.$$
The relation of two cocycles being cohomologous is an equivalence relation on $Z^1_{\d}(\G,G)$. The equivalence classes are called cohomology classes. The equivalence class of the trivial cocycle (the one that maps each $\sigma\in \G$ to the identity of $G$) is denoted by $B^1_{\d}(\G,G)$ and its elements are called definable (1-)coboundaries from $\G$ to $G$. The set of all equivalence classes is called the 1st definable cohomology set of $\G$ with coefficients in $G$ (over $A$), and is denoted by $H^1_{\d}(\G,G)$. 

\begin{remark}
\begin{enumerate}
\item When $M=\operatorname{dcl}(A)$ the cohomology $H^1_{\d}(\G,G)$ is trivial.
\item In general, the first cohomology set $H^1_{\d}(\G,G)$ has the structure of a \emph{pointed set} with distinguished element being $B^1_{\d}(\G,G)$. When $G$ is abelian, the set of definable cocyles $Z^1_{\d}(\G,G)$ has a natural structure of an abelian group (by point-wise multiplication) with $B^{1}_{\d}(\G,G)$ a subgroup; and thus, in this case, $H^1_{\d}(\G,G)$ inherits a group structure given by the factor group $Z^1_{\d}(\G,G)/B^1_{\d}(\G,G)$.
\end{enumerate}
\end{remark}

\begin{example}\label{examples} \
\begin{enumerate}
\item [(i)] \emph{Algebraic Galois cohomology.} In the case when $A=k$ is a perfect field and $M=k^{\text{alg}}$ is the field-theoretic algebraic closure of $A$, the group $G$ is simply the $k^{\text{alg}}$-rational points of an algebraic group over $k$ and, by \cite[Lemma 2.6]{Pillay-cohomology}, we are in the classical situation of (algebraic) Galois cohomology. More precisely, the definable cohomology $H^1_{\d}(\G,G)$ agrees with the Galois cohomology over $k$ with coefficients in $G$ usually denoted $H^1(k,G)$. Most of our statements are inspired from well known results in this context, see for instance \cite{PR} or \cite{Serrebook}.
\smallskip
\item [(ii)] \emph{Differential constrained cohomology.} In the case when $A$ is a differential field of characteristic zero in finitely many commuting derivations $(k,\Delta)$ and $M$ is a differential closure $(k^{\text{diff}},\Delta)$, by \cite{Pillay-groups}, the definable group $G$ is simply the $k^{\text{diff}}$-rational points of a differential algebraic group over $k$. Also, by \cite[\S 3]{Pillay-cohomology}, the definable cohomology $H^1_{\d}(\G,G)$ coincides with Kolchin's constrained cohomology $H^1_{\Delta}(k,G)$ from  \cite[Chapter VII]{Kolchinbook2}. It is worth noting that, while the results of the current paper are well known in classical Galois cohomology, they are as a matter of fact novel in the context of differential constrained cohomology (they do not appear in \cite{Kolchinbook2} or elsewhere to the authors' knowledge).
\end{enumerate}
\end{example}

We recall that an $A$-definable (right) principal homogeneous space for $G$ is an $A$-definable set $X$ together with an $A$-definable regular (also called strictly transitive) right $G$-action. Two such spaces are said to be $A$-definably isomorphic if there is an $A$-definable bijection preserving the $G$-actions. The collection of $A$-definable principal homogeneous spaces, up to $A$-definable isomorphism, is a pointed set (with $G$ as distinguished element) denoted by $P_{\d}(G)$. The main result of \cite{Pillay-cohomology} establishes that as pointed sets $P_{\d}(G)$ and $H^1_{\d}(\G,G)$ are naturally isomorphic.

\medskip

Letting $\Aut_{\d}(G)$ denote the set of ($M$-)definable group automorphisms of $G$, in Section~\ref{formcoho} of this note, we prove that the definable cohomology set $\Hd$ (see Section~\ref{formcoho} for details) is isomorphic, as a pointed set, to the collection of $A$-forms of $G$ up to $A$-definable isomorphism. This result is then used in Section~\ref{finalsection}.

\medskip

Then, in Section~\ref{onsequence}, we are explicit on the construction of exact sequences in definable cohomology from (normal) short exact sequences for $G$. We further discuss the method of twisting definable cohomology by inner automorphisms, in the spirit of Serre's algebraic twisting (cf. \cite[Chapter I, \S5.3]{Serrebook} and \cite{SerreBour}), and utilise it to describe arbitrary fibres of the maps in the (original) cohomology sequence. Recall that our cohomologies are just pointed sets, not necessarily groups, thus the kernel of a morphism does not give us information about the other fibres. 

\medskip

In Section~\ref{finalsection}, we put together the aforementioned results to show that if $N$ is a normal $A$-definable subgroup of $G$ and $\mu\in H^{1}_{\d}(\G,G)$, then there is a \emph{surjection} from $H^{1}_{\d}(\G,N_\mu)$, where $N_\mu$ is the $A$-form of $N$ corresponding to the cohomology class in $H^1_{\d}(\G, \Aut_{\d}(N))$ induced by conjugation of $N$ by $\mu$, to the fibre
$$\P(\mu):= (\pi^1)^{-1}(\pi^1(\mu))$$
where $\pi^1:H^{1}_{\d}(\G,G)\to H^{1}_{\d}(\G,G/N)$ is the induced morphism (of pointed sets) from the canonical projection $\pi:G\to G/N$. As an application/corollary, we prove that if $H_{\d}(\G,G/N)$ is finite and  $H^1_{\d}(\G, N_{\mu})$ is also finite, for all $\mu\in H^{1}_{\d}(\G,G)$, then $H^1_{\d}(\G, G)$ is finite as well. 

\medskip

The corollary above, in the special case of Example 1.3 (ii) above,  is used in \cite[Lemma 2.6]{LS-Pillay}. A sketch of the argument is pointed out there, in the differential case, with details left to the reader.  One of the points of the current paper is to give  a detailed account in the more general model theoretic environment.  In fact, as we point out in Section 5,  one obtains also that  if $H_{\d}(\G,G/N)$ has cardinality $\leq \kappa$ and  $H^1_{\d}(\G, N_{\mu})$ has cardinality $\leq \kappa$, for all $\mu\in H^{1}_{\d}(\G,G)$, then $H^1_{\d}(\G, G)$ has cardinality $\leq\kappa$ too.   In Section 5, we apply this back in the differential context to prove the {\em countability} of $H^{1}_{\Delta}(k,G)$ whenever $k$ is a bounded, differentially large, field and $G$ is an arbitrary differential algebraic group over $k$.

\section{Definable cohomology and $A$-forms of $G$}\label{formcoho}

We carry on the notation and assumptions from the previous section; in particular, those on $M$, $A$, $\G$ and $G$. We let $\Autd(G)$ denote the set of definable (with parameters in $M$) group automorphisms of $G$. Note that this is a group -- with respect to composition. Also note that the automorphism group $\G$ acts naturally on $\Autd(G)$; namely, if $\sigma\in \G$ and $\phi\in\Autd(G)$, then $\sigma(\phi)$ is obtained by applying $\sigma$ to the graph of $\phi$ (a definable set), which yields another definable automorphism of $G$. Furthermore, one readily checks that for $\phi_1,\phi_2\in \Autd(G)$ we have
\begin{equation}\label{compatible}
\sigma(\phi_1\cdot \phi_2)=\sigma(\phi_1)\cdot\sigma(\phi_2).
\end{equation}
In other words, the action of $\G$ in $\Autd(G)$ is by group automorphisms. Furthermore, equipping $\Autd(G)$ with its discrete topology, elimination of imaginaries implies that this action is continuous.

\medskip

Using the action of $\G$ on $\Autd(G)$, we define a \emph{(1-)cocycle} from $\G$ to $\Autd(G)$ to be a map 
$$\Phi:\G\to \Autd(G)$$
such that 
$$\Phi(\sigma_1\cdot \sigma_2)=\Phi(\sigma_1)\cdot \sigma_1(\Phi(\sigma_2)).$$
The trivial group homomorphism from $\G$ to $\Autd(G)$ is a cocycle, called the \emph{trivial} cocycle. Two cocycles $\Phi$ and $\Psi$ from $\G$ to $\Autd(G)$ are said to be cohomologous if there exist $\phi\in \Autd(G)$ such that 
$$\Psi(\sigma)=\phi^{-1}\cdot \Phi(\sigma)\cdot \sigma(\phi) \quad \textrm{ for all } \sigma\in \G.$$
Being cohomologous is an equivalence relation: transitivity follows by \eqref{compatible}. We say that the cocycle $\Phi$ is \emph{definable} if if there exists a tuple $a$ from $M$ and an $A$-definable (partial) function $h(x,y,z)$ such that
$$\Phi(\sigma)(-)=h(a,\sigma(a), -) \quad \textrm{ for all } \sigma\in \G.$$
We note that the trivial cocycle is definable, and if $\Phi$ is a cocycle that is cohomologous to a definable cocycle then $\Phi$ is itself definable. Furthermore, one readily verifies that definable cocycles are continuous. The set of definable cocycles is denoted by $Z^1_{\d}(\G,\Autd(G))$. The set of cohomology classes of $Z^1_{\d}(\G,\Autd(G))$ (under the equivalence relation of being cohomologous) is the 1st definable cohomology of $\G$ with coefficients in $\Autd(G)$ and is denoted by $H^1_{\d}(\G,\Autd(G))$.

\begin{remark}
As in the case of the definable cohomology $H^1_{\d}(\G,G)$, the cohomology $\Hd$ has the structure of a pointed set, with the cohomology class of the trivial cocyle (also called the set of \emph{coboundaries} from $\G$ to $\Autd(G)$) being the distinguished element. Further, in the case when the group $\Autd(G)$ is abelian, $\Hd$ has the structure of an abelian group. This follows from the fact that, in this case, the natural (commutative) group structure on the set of maps from $\G$ to $\Autd(G)$ given by 
$$\Psi_1* \Psi_1(\sigma)=\Psi_1(\sigma)\cdot \Psi_2(\sigma)$$
restricts to a (commutative) group structure on $Z^1_{\d}(\G,\Autd(G))$.
\end{remark}

\medskip

By an \emph{$A$-form of $G$} we mean an  $A$-definable group $H$ which is definably (over $M$) isomorphic to $G$. Let $H$ be an  $A$-form of $G$. Note that, for any definable (group) isomorphism $f:G\to H$ and $\sigma\in \G$, the map $\sigma(f):G\to H$ is again a definable isomorphism. Thus, we have a well-defined map $\Phi_{H,f}:\G\to \Autd(G)$ given by
$$\Phi_{H,f}(\sigma)=f^{-1}\circ \sigma(f)$$
Furthermore, $\Phi_{H,f}$ is a definable cocycle. It is clearly definable and
\begin{align*}
\Phi_{H,f}(\sigma_1\cdot\sigma_2) &=f^{-1}\sigma_1(\sigma_2(f)) \\
&=f^{-1}\sigma_1(ff^{-1}\sigma_2(f)) \\
&=f^{-1}\sigma_1(f)\sigma_1(f^{-1}\sigma_2(f)) \\
&=\Phi_{H,f}(\sigma_1)\cdot \sigma_1(\Phi_{H,f}(\sigma_2))
\end{align*} 
Now, if $g:G\to H$ is another definable isomorphism, then $\Phi_{H,f}$ and $\Phi_{H,g}$ are cohomologous. Indeed, 
$$\Phi_{H,g}(\sigma)=(f^{-1}g)^{-1}\Phi_{H,f}(\sigma)\sigma(f^{-1}g).$$
More generally, if $N$ is another $A$-form of $G$ which is (group) definably isomorphic to $H$ over $A$, say witnessed by $\eta:N\to H$, and $p:G\to N$ is a definable isomorphism, then $\Phi_{H,f}$ and $\Phi_{N,p}$ are cohomologous. Indeed
$$\Phi_{N,p}(\sigma)=(f^{-1}\, \eta\,  p)^{-1}\Phi_{H,f}(\sigma)\sigma(f^{-1}\, \eta\,  p).$$

Let $\mathfrak {F}_{\d}(G)$ denote the set of $A$-forms of $G$ up to $A$-definable isomorphism. The above discussion yields a well-defined map 
$$\mathfrak{F}_{\d}(G)\to \Hd$$
given by mapping $(H,f)$ to the cohomology class of $\Phi_{H,f}$, where $H$ an $A$-form of $G$ and $f:G\to H$ a definable (over $M$) isomorphism. We now show that the displayed map is an isomorphism of pointed sets. The distinguished element in $\mathfrak{F}_{\d}(G)$ being $G$. Clearly, $\Phi_{G,\operatorname{Id}}$ is the trivial cocycle from $\G$ to $\Autd(G)$, hence $(G,\operatorname{Id})$ maps to the cohomology class of the trivial cocycle (i.e., the set of coboundaries). The remaining of this section is devoted to showing that this map is a bijection. 

\begin{remark}\label{setdef}
It follows from our assumptions (see Assumptions~\ref{assumptions}) that a definable set $X$ is $A$-definable if and only $X$ is fixed setwise by any element of $\G$. We will use this to show injectivity of the above map.
\end{remark}

Let us show injectivity of $\mathfrak{F}_{\d}(G)\to \Hd$. Take $(H,f)$ and $(N,p)$ where $H$ and $N$ are $A$-forms of $G$ and $f:G\to H$ and $p:G\to N$ are definable isomorphisms. Assume $\Phi_{H,f}$ is cohomologous to $\Phi_{N,p}$, then there exists $\phi\in \Autd(G)$ such that  for all $\sigma\in \G$
$$\Phi_{N,p}(\sigma)=\phi^{-1}\Phi_{H,f}(\sigma)\,\sigma(\phi).$$ 
In other words, 
$$p^{-1}\sigma(p)=\phi^{-1}f^{-1}\sigma(f) \sigma(\phi),$$
rearranging we get
 $$f\phi p^{-1}=\sigma(f\phi p^{-1})$$
It follows, see Remark~\ref{setdef}, that the isomorphism $f\phi p^{-1}:N\to H$ is $A$-definable. Thus $N$ and $H$ are identified in $\mathfrak{F}_{\d}(G)$, showing the map is indeed injective. 

\medskip

We now show surjectivity. Let $\Phi\in Z^1_{\d}(\G,\Autd(G))$. Then there is tuple $a$ from $M$ and an $A$-definable (partial) function $h(x,y,z)$ such that $\Phi(\sigma)(-)=h(a,\sigma(a),-)$ for all $\sigma\in \G$. 

\begin{lemma}\label{prop}
Let $a$ and $h$ be as above. For any $b, c, d$ tuples from $M$ realising $tp(a/A)$ we have that
$h(b,c,-)\in\Autd(G)$. Furthermore, 
$$h(b,c,-)=h(b,d,h(d,c,-)).$$
It follows that $h(b,b,-)=\operatorname{Id}_G$ and $h(b,c,-)^{-1}=h(c,b,-).$
\end{lemma}
\begin{proof}
Let $\sigma,\tau,\eta\in \G$ be such that 
$$\sigma(a)=b,\;  \sigma\tau\eta(a)=c\;  \textrm{ and }\; \sigma\tau(a)=d.$$ 
Then
\begin{align*}
h(b,c,-)&=h(\sigma(a), \sigma\tau\eta(a),-) \\
&=\sigma(\Phi(\tau\eta))(-) \\
&=\sigma(\Phi(\tau)\tau(\Phi(\eta)))(-) \\
&=h(\sigma(a),\sigma\tau(a), h(\sigma\tau(a),\sigma\tau\eta(a),-)) \\
&=h(b,d,h(d,c,-)).
\end{align*}
\end{proof}

Now let $Y$ be the set of realisations of $tp(a/A)$ in $M$. Then $Y$ is an $A$-definable set. Set $Z=Y\times G$. Define a relation $R$ on $Z$ given by 
$$(b,d)\,R\, (c,e) \; \iff \; h(b,c,e)=d$$
It follows, using Lemma~\ref{prop}, that $R$ is an equivalence relation (which is of course $A$-definable). Let $H=Z/R$. Then, by elimination of imaginaries, we identify the quotient $H$ with an $A$-definable set. We now equip $H$ with a $A$-definable group structure. Define, for $(b,d), (c,e)\in Z$,
$$(b,d)*(c,e)=(b,d\cdot h(b,c,e))\in Z$$
Using Lemma~\ref{prop}, one readily checks that this binary operation $*$ on $Z$ is $R$-invariant and thus induces a $A$-definable binary operation on $H$ that we also denote by $*$. We show that this is in fact a group structure on $H$. We prove this by showing that there is a definable (over $M$) bijection $f$ between $H$ and $G$ such that the induced group structure on $H$ is precisely $*$. Note that this further shows that $H$ is a $A$-form of $G$.

Let $f:Z\to G$ be defined by $f(b,d)=h(a,b,d)$. Since $h(a,b,-)\in \Autd(G)$, the function $f$ is clearly surjective. Furthermore, again using Lemma~\ref{prop}, one checks that
$$f(a,b,d)=f(a,c,e) \; \iff \; h(b,c,e)=d \; \iff\; (b,d) \, R\, (c,e).$$
Thus $f$ induces a definable (over $M$) bijection from $H=Z/R$ to $G$, which we also denote by $f$. Now, using Lemma~\ref{prop}, the induced group structure on $H$ via this bijection is 
\begin{align*}
f^{-1}(f(b,d)\cdot f(c,e)) &=f^{-1}(h(a,b,d)\cdot h(a,c,e)) \\
&=f^{-1}(h(a,b,d)\cdot h(a,b,h(b,c,e))) \\
&=f^{-1}(h(a,b,d\cdot h(b,c,e)) \\
&=f^{-1}(f(b,d\cdot h(b,c,e)))  \\
&=(b,d\cdot h(b,c,e)) \\
&=(b,d)*(c,e).
\end{align*}
We have thus shown that $(H,*)$ is a $A$-form of $G$ (witnessed by $f^{-1}:G\to H$). We claim that $(H,f^{-1})$ maps to $\Phi$ (namely $\Phi=\Phi_{H,f^{-1}}$). Indeed, for any $\sigma\in \G$, using once again Lemma~\ref{prop}, we have
$$\Phi(\sigma)\sigma(f)(b,d)=h(a,\sigma(a), \sigma(f)(b,d))=h(a,\sigma(a), h(\sigma(a), b,d))=h(a,b,d)=f(b,d).$$
and so 
$$\Phi(\sigma)=f\cdot \sigma(f^{-1})$$
this shows that $\Phi=\Phi_{H,f^{-1}}$, as desired.

\medskip

We have thus shown:

\begin{proposition}\label{isoform}
There is a natural isomorphism (as pointed sets) between $\mathfrak{F}_{\d}(G)$ and $\Hd$.
\end{proposition}

\section{Exact sequences and twisting in definable cohomology}\label{onsequence}

We carry forward the conventions and assumptions on $M$, $A$, $\G$ and $G$ from previous sections. Let $X$ be an $A$-definable set in the structure $M$. Then $\G$ acts naturally on $X$. The 0th definable cohomology from $\G$ to $X$, denoted $H^0_{\d}(\G,X)$, is the set of $\G$-invariant points of $X$. By our assumptions (see Assumptions~\ref{assumptions}), this coincides with the points of $X$ whose entries are all in $\operatorname{dcl}(A)$. In the case when $X$ has in addition the structure of an $A$-definable pointed set (i.e., has a distinguished point which is $\G$-invariant and thus in $\operatorname{dcl}(A)$), then $H^0_{\d}(\G,X)$ inherits the structure of a pointed set sharing its distinguished element with $X$. Furthermore, when $X=G$ (an $A$-definable group), the 0th definable cohomology $H^0_{\d}(\G,G)$ is clearly a subgroup of $G$.

Now consider an $A$-definable map $f:X\to X'$ of $A$-definable sets. Since $f$ preserves $\G$-invariants, it induces a map in cohomology
$$f^0:H^0_{\d}(\G,X)\to H^0_{\d}(\G,X').$$
In the case that $X$ and $X'$ have in addition the structure of $A$-definable pointed sets (the distinguished elements being $\G$-invariant) and $f$ is a homomorphism of pointed sets (i.e., maps the distinguished element of $X$ to that of $X'$), the above map $f^0$ is a homomorphism of pointed sets. Furthermore, when $X=G$ and $X'=G'$ are $A$-definable groups and $f$ is a group homomorphism, the map 
$$f^0:H^0_{\d}(\G,G)\to H^0_{\d}(\G,G')$$ 
is a group homomorphism.

Recall, from the introduction, that $Z^1_{\d}(\G,G)$ and $H^1_{\d}(\G,G)$ denote the sets of definable cocycles and cohomology from $\G$ to $G$, respectively. Let $f:G\to G'$ be an $A$-definable group homomorphism (between the $A$-definable groups $G$ and $G'$). For any definable 1-cocycle $\Phi\in Z^1_{\d}(\G,G)$, the composition $f\circ \Phi:\G\to  G'$ is again a definable 1-cocycle. Furthermore, if $\Psi$ is cohomologous to $\Phi$, then $f\circ\Phi$ is cohomologous to $f\circ \Psi$ as elements of $Z^1_{\d}(\G,G')$. Indeed, suppose $b\in G$ is such that $\Phi(\sigma)=b^{-1}\Psi(\sigma)\sigma(b)$ for all $\sigma\in \G$, then one easily verifies
$$(f\circ \Psi)(\sigma)=(f(b))^{-1}\, (f\circ \Psi)(\sigma)\, \sigma(f(b)).$$
Thus, the asssigment $\Psi\mapsto (f\circ \Phi)$ defines a map from $Z^1_{\d}(\G,G)$ to $Z^1_{\d}(\G,G')$ that induces a map in cohomology
$$f^1:H^1_{\d}(\G,G)\to H^1_{\d}(\G,G')$$
Moreover, $f^1$ is a morphism of pointed sets (meaning it preserves distinguished elements), and when $G$ and $G'$ are abelian then it is a group homomorphism. 

\medskip

\subsection{Exact sequence in cohomology}\label{exactcoho} We now fix an arbitrary (not necessarily normal) $A$-definable subgroup $N$ of $G$. This yields the natural short exact sequence of $A$-definable pointed sets
\begin{equation}\label{seq}
1\to N \xrightarrow{\iota} G \xrightarrow{\pi} G/N\to 1.
\end{equation}
Note that, since $N$ is not necessarily normal in $G$, the left-coset space $G/N$ is not necessarily a group. It is however an $A$-definable (left) homogeneous space for $G$ (i.e., $G/N$ is an $A$-definable set equipped with an $A$-definable transitive left $G$-action -- in this case the natural $G$-action on the left). While the inclusion map $\iota$ is a group homomorphism, the projection map $\pi$ is generally just a morphism of $A$-definable homogenous $G$-spaces. Furthermore, $G/N$ has a natural structure of an $A$-definable pointed set, with distinguished element being the coset $N$, and so $\pi$ is also a homomorphism of $A$-definable pointed sets. When $N$ is normal in $G$, the above sequence is a normal short exact sequence of $A$-definable groups.

\medskip

Let $a\in H^0_{\d}(\G,G/N)$, here we think of $G/N$ as a pointed set and so the 0th cohomology (which the same as the $\G$-invariants of $G/N$) is also a pointed set as explained above. Then, $a$ is of the form $\pi(b)$ for some $b\in G$. For any $\sigma\in \G$ we have 
$$\pi(\sigma(b))=\sigma(\pi(b))=\sigma(a)=a=\pi(b)$$
and so $b^{-1}\cdot \sigma(b)\in N$. Thus the assignment $\sigma\mapsto b^{-1}\cdot \sigma(b)$ defines a map from $\G$ to $N$. Moreover, this map clearly belongs to $Z^1_{\d}(\G,N)$; and if $b'$ is another element of $G$ with $\pi(b')=a$, then $b'=bc$ for some $c\in N$ and hence
$$(bc)^{-1}\cdot \sigma(bc)=c^{-1}(b^{-1}\cdot \sigma(b))\sigma(c).$$
This shows that the assignment $\sigma\mapsto (b')^{-1}\cdot \sigma(b')$ is cohomologous to the previous one. Therefore, the cohomology class of this assignment is uniquely determined by $a$, and this class is denoted by $\delta(a)\in H^1_{\d}(\G,N)$. We have now a morphism
$$\delta : H^0_{\d}(\G,G/N)\to H^1_{\d}(\G,N)$$
of pointed sets (and of groups when $G$ is abelian) called the connecting homomorphism.

\bigskip

We are now in the position to state (and prove) how exactness transfers when passing to the sequence in cohomology.

\begin{theorem}\label{cohoseq}
Suppose $N$ is an $A$-definable subgroup of $G$ (not necessarily normal). Then, the sequence in cohomology induced from \eqref{seq}
$$1\to H^{0}_{\d}(\G,N)\xrightarrow{\iota^0}H^0_{\d}(\G.G)\xrightarrow{\pi^0} H^0_{\d}(\G,G/N)\xrightarrow{\delta}H^1_{\d}(\G,N)\xrightarrow{\iota^1} H^{1}_{\d}(\G,G)$$
is exact (as a sequence of pointed sets). Furthermore, if $N$ is normal in $G$, then the sequence remains exact when extended by
$$H^{1}_{\d}(\G,G)\xrightarrow{\pi^1}H^1_{\d}(\G,G/N).$$
\end{theorem}
\begin{proof}
It is easy to check that the sequence is exact at $H^{0}_{\d}(\G,N)$ and at $H^0_{\d}(\G,G)$. We now prove it is exact at $H^0_{\d}(\G,G/N)$. Let $a\in H^0_{\d}(\G,G/N)$ and $b\in G$ such that $\pi(b)=a$. Then, $\delta(a)=1$ if and only if there is $c\in N$ such that
$$b^{-1}\sigma(b)=c^{-1}\sigma(c), \quad \text{ for all }\sigma\in \G;$$
that is, $bc^{-1}\in H^0_{\d}(\G,G)$. Thus, $\delta(a)=1$ if and only there is $d\in H^0_{\d}(\G,G)$ such that $\pi^0(d)=a$. 

We now check the sequence in exact at $H^1_{\d}(\G,N)$. Let $\eta\in H^1_{\d}(\G,N)$ and $\Phi$ a cocycle in $Z^1_{\d}(\G,N)$ with cohomology class $\eta$. Then, $\iota^1(\eta)=1$ if and only there is $b\in G$ such that $\Phi(\sigma)=b^{-1}\sigma(b)$ for all $\sigma\in \G$. The latter implies 
$$\pi(b)^{-1}\sigma(\pi(b))=\pi(b^{-1}\sigma(b))=\pi(\Phi(\sigma))=1$$
and so $\pi(b)\in H^0_{\d}(\G,G/N)$. Clearly, $\delta(\pi(b))=\eta$, and exactness at $H^1_{\d}(\G,N)$ follows. 

Finally, assuming $N$ is normal in $G$, we check exactness at $H^1_{\d}(\G,G)$. Let $\Phi\in Z^1_{\d}(\G,N)$. Then, since $im(\iota)=ker(\pi)$, we get that $\pi\circ \iota\circ \Phi$ is the trivial cocycle from $\G$ to $G/N$. It follows that $im(\iota^1)\subseteq ker(\pi^1)$. Now, let $\Psi\in Z^1_{\d}(\G,G)$ be such that $\pi\circ \Psi$ is a coboundary from $\G$ to $G/N$. This means that there is $d\in G/N$ such that 
$$\pi(\Psi(\sigma))=d^{-1}\cdot \sigma(d), \quad \text{ for all } \sigma\in \G.$$
Let $a\in G$ such that $\pi(a)=d$. Then we can rearrange the above to
$$\pi(a\cdot  \Psi(\sigma)\cdot \sigma(a)^{-1})=1.$$
Since $\iota$ is injective and $im(\iota)=ker(\pi)$, there is a unique $c\in N$ such that $\iota(c)=a\cdot  \Psi(\sigma)\cdot \sigma(a)^{-1}$. If we set $\Lambda:\G\to N$ to be $\Lambda(\sigma)=c$ where $c$ is the unique $c$ just found, then one can readily check that $\Lambda$ is a definable cocycle from $\G$ to $N$. Moreover, $\iota\circ \Lambda$ is cohomologous to $\Psi$. This proves $ker(\pi^1)\subseteq im(\iota^1)$, as desired. 
\end{proof}

\smallskip

\subsection{Twisting definable cohomology}\label{twist} Recall that when $G$ is abelian, the cohomology sequence in Theorem~\ref{cohoseq} is an exact sequence of groups and group homomorphism, and so fibres are just translates of the kernels. However, in the general situation (when $G$ is not necessarily abelian), the exact sequence is just a sequence of pointed sets and thus carry less information than in the abelian case. Indeed, the kernel of a morphism of pointed sets does not generally yield any information about the other fibres. 

As in algebraic Galois cohomology \cite[\S1.3]{PR}, this can be overcome using a method introduced by Serre \cite[\S1.5]{SerreBour} called \emph{twisting}. We now discuss how a special case (which suffices for our purposes) of this method transfers to our general definable model-theoretic context. Let $\Phi\in Z^1_{\d}(\G,G)$ and define a new (twisted) action of $\G$ on $G$ by
$$\sigma * g=C_{\Phi(\sigma)}(\sigma(n))$$
where, for any $h\in G$, $C_h:G\to G$ denotes conjugation by $h$; namely, $C_{h}(g)=h\cdot g\cdot h^{-1}$ for all $g\in G$. We call this the induced action of $\G$ on $G$ twisted via conjugation by $\Phi$. When clarity is needed we denote this action by saying that $\G_\Phi$ acts on $G$. One readily checks that this is indeed a group action and in fact $\G_\Phi$ acts on $G$ by group automorphisms. Furthermore, since $\Phi:\G\to G$ is continuous, this action is continuous. 

With respect to this (non-natural) action of $\G_\Phi$ on $G$, we can define the notion of definable  (1-)cocycles, denoted $Z^1_{\d}(\G_\Phi,G)$, as those maps $\Psi:\G\to G$ such that
$$\Psi(\sigma_1\sigma_2)=\Psi(\sigma_1)\cdot \sigma_1 *(\Psi(\sigma_2)).$$
As before, such cocycle is said to be definable if there is tuple $a$ from $M$ and an $A$-definable function $h(x,y)$ such that 
$$\Psi(\sigma)=h(a,\sigma(a)), \quad \text{ for all } \sigma\in G.$$
Thus definable cocycles are still continuous (in this twisted context). Two cocycles $\Psi$ and $\Lambda$ in $Z^1_{\d}(\G_\Phi,G)$ are cohomologous if there is $b\in G$ such that
$$\Psi(\sigma)=b^{-1}\cdot \Lambda(\sigma)\cdot (\sigma * b), \quad \text{ for all }\sigma\in \G.$$
This again yields an equivalence relation. The set of equivalence classes is the 1st definable cohomology set $H^{1}_{\d}(\G_\Phi,G)$ twisted by $\Phi$. This has a natural structure of a pointed set. 

\begin{remark}
Note that when $G$ is abelian the twisted action of $\G_\Phi$ on $G$ is just the natural action of $\G$ on $G$ and so the cohomologies $H^{1}_{\d}(\G_\Phi,G)$ and $H^1_{\d}(\G,G)$ coincide.
\end{remark}

As with the 0th definable cohomology set $H^0_{\d}(\G,G)$, the 0th definable twisted cohomology $H^0_{\d}(\G_\Phi,G)$ is defined as the set of $\G_\Phi$-invariant points of $G$, which is a subgroup of $G$ (as $\G_\Phi$ acts on $G$ by group automorphisms). 

Now let $N$ be an $A$-definable normal subgroup of $G$. We can then restrict the action of $\G_\Phi$ on $G$ to $N$ (since $N$ is normal in $G$). This yields a continuous (twisted) action by group automorphisms of $\G_\Phi$ on $N$. In a similar way as above, we can construct the sets of twisted definable cocycles and cohomology $Z^1_{\d}(\G_\Phi,N)$ and $H^1_{\d}(\G_\Phi,N)$, respectively. Furthermore, letting $\pi:G\to G/N$ be the natural projection, we can twist the natural action of $\G$ on $G/N$ by the cocycle $\pi\circ \Phi\in Z^1_{\d}(\G,G/N)$. Then one can construct the twisted definable cocycles and cohomology sets $Z^1_{\d}(\G_{\pi\Phi},G/N)$ and $H^1_{\d}(\G_{\pi\Phi},G/N)$. For convenience, and it should not cause confusion, we denote $\G_{\pi\Phi}$ (acting on $G/N$) simply by $\G_\Phi$ and the set of definable cocycles and cohomology by $Z^1_{\d}(\G_{\Phi},G/N)$ and $H^1_{\d}(\G_{\Phi},G/N)$

\medskip

Consider the natural short exact sequence of $A$-definable groups (recall that we are assuming that $N$ is normal in $G$)
$$1\to N \xrightarrow{\iota} G \xrightarrow{\pi} G/N\to 1.$$
With the respect to the above twisted actions of $\G_\Phi$ on $N$, $G$ and $G/N$, one readily checks that for $\sigma\in \G$, $n\in N$, and $g\in G$ we have
$$\iota(\sigma*n)=\sigma*\iota(n) \quad \text{ and }\quad \pi(\sigma*g)=\sigma*\pi(g).$$
It follows from these equalities, exactly as we did in the previous sub-section, that we can induce morphisms of pointed sets in cohomology
$$\iota^i_\Phi:H^i_{\d}(\G_\Phi,N)\to H^i_{\d}(\G_\Phi,G) \quad \text{ and }\quad \pi^i_\Phi:H^i_{\d}(\G_\Phi,G)\to H^i_{\d}(\G_\Phi,G/N)$$
for $i=0,1$. Furthermore, we get a connecting morphism (or pointed sets)
$$\delta_\Phi:H^0(\G_\Phi,G/N)\to H^1_{\d}(\G_\Phi,N),$$
and obtain, by adapting the argument in the proof of Theorem~\ref{cohoseq}, the (twisted) exact sequence
\begin{align*}
1\to H^{0}_{\d}(\G_\Phi,N)\xrightarrow{\iota^0_\Phi}H^0_{\d}(\G_\Phi.G) & \xrightarrow{\pi^0_\Phi} H^0_{\d}(\G_\Phi,G/N)  \\ 
& \xrightarrow{\delta}H^1_{\d_\Phi}(\G_\Phi,N)\xrightarrow{\iota^1_\Phi} H^{1}_{\d}(\G_\Phi,G)\xrightarrow{\pi^1_\Phi}H^1_{\d}(\G_\Phi,G/N).
\end{align*}
of pointed sets.

\medskip

While one is not generally able to multiply cocycles in $Z^1_{\d}(\G,G)$, unless $G$ is abelian say, by going to the twisted context this is possible. We make this precise in the following result which also shows how the cohomologies $H^1_{\d}(\G,G)$ and $H^1_{\d}(\G_\Phi,G)$ are related (in fact in bijection).

\begin{proposition}\label{multiply}
Let $f_\Phi:Z^1_{\d}(\G_\Phi,G)\to Z^1_{\d}(\G,G)$ be the map defined by $f_\Phi(\Psi)=\Psi\cdot \Phi$; namely, for $\sigma\in G$, $f_\Phi(\Psi)(\sigma)=\Psi(\sigma)\cdot\Phi(\sigma)$. Then, $f_\Phi$ is a bijection. Furthermore, it induces a bijection in cohomology $F_\Phi:H^1_{\d}(\G_\Phi,G)\to H^1_{\d}(\G,G)$ which maps the distinguished element to the cohomology class of $\Phi$.
\end{proposition}
\begin{proof}
We first prove that indeed $f_\Phi(\Psi)\in Z^{1}_{\d}(\G,G)$ for every $\Psi\in Z^1_{\d}(\G_\Phi,G)$. 

Indeed, it is clearly definable (since $\Phi$ and $\Psi$ are definable) and for $\sigma_1,\sigma_2\in \G$ we have
\begin{align*}
f_\Phi(\Psi)(\sigma_1\sigma_2) &= \Psi(\sigma_1\sigma_2)\; \Phi(\sigma_1\sigma_2)\\
&= \Psi(\sigma_1)\sigma_1*(\Psi(\sigma_2))\; \Phi(\sigma_1)\sigma_1(\Phi(\sigma_2)) \\
&= \Psi(\sigma_1) \Phi(\sigma_1)\sigma_1(\Psi(\sigma_2))\Phi(\sigma_1)^{-1}\Phi(\sigma_1)\sigma_1(\Phi(\sigma_2)) \\
&= f_\Phi(\Psi)(\sigma_1)\cdot \sigma_1(f_\Phi(\Psi)(\sigma_2)).
\end{align*}
One easily checks that $f_\Phi$ has inverse $f^{-1}_\Phi:Z^1_{\d}(\G,G)\to Z^{1}_{\d}(\G_\Phi,G)$ given by 
$$f_\Phi^{-1}(\Lambda)(\sigma)=\Lambda(\sigma)\cdot (\Phi(\sigma))^{-1}.$$
Now, to pass to cohomology, let $\Psi,\Lambda\in Z^1_{\d}(\G_\Phi,G)$ be cohomologous witnessed by $b\in G$; namely,
$$\Lambda(\sigma)=b^{-1}\cdot\Psi(\sigma) \cdot \sigma* b, \quad \text{ for all }\sigma \in \G.$$
Multiplying on the right by $\Phi(\sigma)$ we get hat 
$$f_\Phi(\Lambda)(\sigma)=b^{-1}\cdot f_\Phi(\Psi)(\sigma)\cdot \sigma(b).$$
Thus, $f_\Phi(\Lambda)$ and $f_\Phi(\Psi)$ are cohomologous as elements in $Z^1_{\d}(\G,G)$, and so $f_\Phi$ induces a well defined map $F_\Phi:H^1_{\d}(\G_\Phi,G)\to H^{1}_{\d}(\G,G)$. Let $1$ be the trivial cocycle in $Z^1_{\d}(\G_\Phi,G)$, then, by definition, $F_\Phi$ maps the distinguished element of $H^1_{\d}(\G_\Phi,G)$ to the cohomology class of
$$f_\Phi(1)=\Phi,$$
as desired. Finally, one can similarly check that the inverse function $f^{-1}_\Phi$ induces a well defined map in cohomology $$F_\Phi^{-1}:H^1_{\d}(\G,G)\to H^1_{\d}(\G_\Phi,G)$$ which an easy computation shows that it is in fact the inverse for $F_\Phi$.
\end{proof}
 
From the proof of the proposition, we see that the bijection 
$$F_\Phi:H^1_{\d}(\G_\Phi,G)\to H^1_{\d}(\G,G)$$ 
restricts to a bijection between the kernel of $\pi^1_\Phi$ and the fibre $(\pi^1)^{-1}(\pi^1(\mu))$ where $\mu$ is the cohomology class of $\Phi$. Thus, we have proved the following.

\begin{corollary}\label{fibres}
With the above notation and $\mu$ denoting the cohomology class of $\Phi$, there is a natural bijection between $ker(\pi^1_\Phi)$ and the fibre
$$\P(\mu):=(\pi^1)^{-1}(\pi^1(\mu)).$$
\end{corollary}

This demonstrates how, using the method of twisting, one can describe the fibres of $\pi^1$ in the original (non-twisted) cohomology sequence.

\begin{remark}\label{welldefined}
We note that if $\Psi$ is a cocycle in $Z^1_{\d}(\G,G)$ which is cohomologous to $\Phi$, witnessed by $b\in G$ (i.e., $\Phi(\sigma)=b^{-1}\Psi(\sigma)\sigma(b)$), then the inner automorphism of $G$ given by conjugation $C_b:G\to G$ commutes with the twisted actions of $G_\Phi$ and $G_\Psi$ on $G$. More precisely, 
$$C_b(\sigma*_\Phi g)=\sigma*_{\Psi}C_b(g), \quad \text{ for all }\sigma\in \G, g\in G.$$
It follows that there are natural isomorphisms of pointed sets
$$Z^1_{\d}(\G_\Phi,G)\to Z^1_{\d}(\G_\Psi,G) \quad \text{ and }\quad H^i_{\d}(\G_\Phi,G)\to H^i_{\d}(\G_\Psi,G)$$
for $i=0,1$. Thus, for any $\mu\in H^1_{\d}(\G,G)$, we may and will denote by $Z^1_{\d}(\G_\mu,G)$ and $H^i_{\d}(\G_\mu,G)$ the set of definable cocycles $Z^1_{\d}(\G_\Phi,G)$ and cohomology sets $H^i_{\d}(\G_\Phi,G)$ for $i=0,1$, respectively, with $\Phi$ any definable cocycle having cohomology class $\mu$. Similar observations apply to $Z^1_{\d}(\G_\mu,N)$ and $H^i_{\d}(\G_\mu,N)$, and to $Z^1_{\d}(\G_\mu,G/N)$ and $H^i_{\d}(\G_\mu,G/N)$, where $N$ is an $A$-definable normal subgroup of $G$.
\end{remark}

%\medskip

\section{Further remarks on fibres and a finiteness result}\label{finalsection}

We fix, throughout this section, an $A$-definable normal subgroup $N$ of $G$. Let $\mu\in H^{1}_{\d}(\G,G)$. In this section we deploy the results from the previous two sections to establish an isomorphism of pointed sets between the twisted cohomology $H^1_{\d}(\G_\mu,N)$ (see Remark~\ref{welldefined}) and the non-twisted $H^1_{\d}(\G,N_\mu)$ where $N_\mu$ is the natural $A$-form of $N$ corresponding to $\mu$ (see details below). We then conclude with a finiteness result in definable cohomology.

But first we note an easy consequence of Corollary~\ref{fibres}. Consider, once again, the short exact sequence of $A$-definable groups 
$$1\to N \xrightarrow{\iota} G \xrightarrow{\pi} G/N\to 1.$$
For $\mu\in H^{1}_{\d}(\G,G)$, we let
$$\P(\mu):=(\pi^1)^{-1}(\pi^1(\mu))$$
denote the fibre of $\pi^1:H^1_{\d}(\G,G)\to H^1_{\d}(\G,G/N)$ above $\pi^1(\mu)$. Then, $\P(\mu)$ has the structure of a pointed set by specifying $\mu$ as its distinguished element. In the next statement we use the notation discussed in Remark~\ref{welldefined}.

\begin{lemma}\label{surjection}
For each $\mu\in H^{1}_{\d}(\G,G)$, there is a natural surjective morphism of pointed sets from $H^{1}_{\d}(\G_\mu,N)$ to $\P(\mu)$.
\end{lemma}
\begin{proof}
With $\Phi\in Z^1_{\d}(\G,G)$ having cohomology class $\mu$, it follows from \S\ref{twist} that $ker(\pi^1_\Phi)$ is contained in the image of $H^1_{\d}(\G_\Phi,N)$ under $\iota^1_\Phi$. By Corollary~\ref{fibres}, this kernel is naturally in bijection with $\P(\mu)$. The result now follows by noting that $\iota^1_\Phi$ maps the distinguished element of $H^1_{\d}(\G_\Phi,N)$ to that of $H^1_{\d}(\G_\Phi,G)$ and, in turn, the map $F_\Phi$ from Proposition~\ref{multiply} maps the latter distinguished element to $\mu$. 
\end{proof}

For $\mu\in H^{1}_{\d}(\G,G)$, let $\Phi\in Z^1_{\d}(\G,G)$ be a cocycle with cohomology class $\mu$. Using the notation of \S\ref{formcoho}, consider the map
$\Lambda_\Phi:\G\to \Autd(N)$ given by 
$$\Lambda_\Phi(\sigma)(-)=C_{\Phi(\sigma)}(-)$$
where recall that $C_{h}:N\to N$ denotes conjugation by $h\in G$; namely, $C_{h}(n)=h\cdot n\cdot h^{-1}$ for all $n\in N$. Since $N$ is a normal in $G$, the image of $\Lambda_\Phi$ is  indeed in $\Autd(N)$. Furthermore, $\Lambda_\Phi$ is a definable cocycle from $\G$ to $\Autd(N)$. It is clearly definable and for $\sigma_1,\sigma_2\in \G$ we have
\begin{align*}
\Lambda_\Phi(\sigma_1\sigma_2)(-) & =C_{\Phi(\sigma_1\sigma_2)}(-) \\
& = \Phi(\sigma_1\sigma_2)\cdot (-)\cdot (\Phi(\sigma_1\sigma_2))^{-1} \\
& = \Phi(\sigma_1)\sigma_1(\Phi(\sigma_2))\cdot (-) \cdot (\sigma_1(\Phi(\sigma_2)))^{-1}(\Phi(\sigma_1))^{-1} \\
& = \Phi(\sigma_1) \cdot \sigma_1(C_{\Phi(\sigma)})(-) \cdot (\Phi(\sigma_1))^{-1} \\
& = C_{\Phi(\sigma_1)}(\sigma_1(C_{\Phi(\sigma_2)})(-)) \\
& = \Lambda_\Phi(\sigma_1) \cdot \sigma_1(\Lambda_\Phi(\sigma_2))(-).
\end{align*}
Furthermore, if $\Psi\in H^1_{\d}(\G,G)$ is cohomologous to $\Phi$, we have that $\Lambda_\Psi$ is cohomologous to $\Lambda_{\Phi}$ in $Z^1_{\d}(\G,\Autd(N))$. Indeed, if $b\in G$ witnesses the fact that $\Psi$ and $\Phi$ are cohomologous, then conjugation $C_{b}\in\Autd(N)$ witnesses that $\Lambda_{\Psi}$ is cohomologous to $\Lambda_{\Phi}$.
 
We thus have a well defined map from $H^1_{\d}(\G,G)$ to $H^{1}_{\d}(\G,\Autd(N))$ mapping $\mu$ to the cohomology class of $\Lambda_\Phi$. By Proposition~\ref{isoform}, this yields a map from $H^1_{\d}(\G,G)$ to $\mathfrak{F}_{\d}(G)$ (the collection of $A$-forms of $N$ up to $A$-definable isomorphism). For $\mu\in H^1_{\d}(\G,G)$ we denote the associated $A$-form of $N$ by $N_\mu$. We note that when $G$ is abelian $\Lambda_\Phi$ is the trivial cocycle and so $N_\mu$ is definably isomorphic to $N$ over $A$.

\medskip

We now state and prove the main result of this section.

\begin{theorem}\label{twoiso}
Let $\mu\in H^1_{\d}(\G,G)$. Set $N_\mu$ to be the $A$-form of $N$ associated to $\mu$, as discussed above. Then, there is natural isomorphism of pointed sets between the twisted cohomology $H^1_{\d}(\G_\mu,N)$ and non-twisted cohomology $H^1_{\d}(\G,N_\mu)$. Note that the latter has coefficients in the $A$-form $N_\mu$.
\end{theorem}

Let $f: N_\mu\to N$ be the definable isomorphism (over $M$) constructed in \S\ref{formcoho} (right before Proposition~\ref{isoform}). We briefly spell out the construction (specialized to the current setting). Let $\Phi\in Z^1_{\d}(\G,G)$ be a definable cocycle with cohomology class $\mu$. Let $a$ be a tuple from $M$ and $h(x,y)$ an $A$-definable function witnessing that $\Phi$ is definable; that is, $\Phi(\sigma)=h(a,\sigma(a))$ for all $\sigma\in \G$. Recall that above we defined $\Lambda_\Phi:\G\to \Autd(N)$ by $\Lambda_\Phi(\sigma)(-)=C_{\Phi(\sigma)}(-)$ and showed it is a definable cocycle from $\G$ to $\Autd(N)$. It follows that
$$\Lambda_\Phi(\sigma)(-)=h(a,\sigma(a))\cdot (-)\cdot  h(a,\sigma(a))^{-1}.$$
In other words, if we set $\tilde h$ to be the $A$-definable (partial) function 
$$\tilde h(x,y,z)=h(x,y)\cdot z\cdot h(x,y)^{-1}$$ 
we get
$$\Lambda_{\Phi}(\sigma)(-)=\tilde h(a,\sigma(a),-), \quad \text{ for all }\sigma\in \G,$$
that is, $a$ and $\tilde h(x,y,z)$ witness that the cocycle $\Lambda_\Phi$ is definable. Now let $Y$ be the realisations of $tp(a/A)$ in $M$, then 
$$N_\mu=(Y\times N)/R$$
where $R$ is the equivalence relation 
$$(b,d)\,R\, (c,e)\quad  \iff \quad \tilde h(b,c,e)=d \quad \iff \quad h(b,c)\cdot e \cdot h(b,c)^{-1}=d.$$
The group structure on $N_\mu$ is the one induced by
$$(b,d) * (c,e)=(b,d\cdot \tilde h(b,c,e))=(b,d\cdot h(b,c)\cdot e \cdot h(b,c)^{-1})$$ 
for $(b,c), (c,e)\in Y\times N$. Finally, the function $f:N_{\mu}\to N$ is given by 
$$f(m)=\tilde h(a,\tau(a),n)=\Lambda_\Phi(\tau)(n)=C_{\Phi(\tau)}(n)$$
where $m$ is the $R$-equivalence class of $(\tau(a),n)\in Y\times N$.

\medskip

The upshot is that now Theorem \ref{twoiso} is a consequence of the following lemma.

\begin{lemma}
Let $f:N_\mu\to N$ be as above. Set 
$$t_f:Z^1_{\d}(\G,N_\mu)\to Z^{1}_{\d}(\G_\mu,N)$$ 
to be the map defined by $t_f(\Psi)=f\circ \Psi$. Then, $t_f$ is a bijection. Furthermore, it induces an isomorphism of pointed sets $T_f:H^1_{\d}(\G,N_\mu)\to H^{1}_{\d}(\G_\mu,N)$
\end{lemma}
\begin{proof}
We first show that $f$ commutes with the actions of $\G$ on $N_\mu$ and $\G_\mu$ on $N$; that is, we prove that for $\sigma\in \G$ and $m\in N_\mu$ we have
\begin{equation}\label{equalaction}
f(\sigma(m))=\sigma* f(m).
\end{equation}
Indeed, let $m$ be the $R$-equivalence class of $(\tau(a), n)\in Y\times N$, by construction of $f$ we have (see discussion above)
$$f(m)=C_{\Phi(\tau)}(n)\quad \text{ and }\quad f(\sigma(m))=C_{\Phi(\sigma\tau)}(\sigma(n)).$$
The latter follows as $\sigma(m)$ is the $R$-equivalence class of $(\sigma\tau(a),\sigma(n))$. Using this, we get
\begin{align*}
\sigma * (f(m))&= \sigma* (C_{\Phi(\tau)}(n)) \\
&= C_{\Phi(\sigma)}\left(\sigma\left(C_{\Phi(\tau)}(n)\right)\right) \\
&= C_{\Phi(\sigma)}C_{\sigma(\Phi(\tau))}\sigma(n) \\
&= C_{\Phi(\sigma\tau)}\sigma(n) \\
&= f(\sigma(m)) 
\end{align*}
as desired. 

Let $\Psi\in Z^1_{\d}(\G,N_\mu)$. It follows from \eqref{equalaction}, as we have done in previous sections, that $f\circ\Psi$ is a cocycle from $\G_\mu$ to $N$.  As $f$ is a definable function, $t_f(\Psi)=f\circ\Psi$ is in fact in $Z^1_{\d}(\G_\mu,N)$. One can analogously consider $f^{-1}:N\to N_\mu$ and easily argue that the map
$$t_{f^{-1}}:Z^1_{\d}(\G_\mu,N)\to Z^{1}_{\d}(\G,N_\mu)$$
given by $t_{f^{-1}}(\Lambda)=f^{-1}\circ \Lambda$ is the inverse of $t_f$. Note here that one can use \eqref{equalaction} to argue that $f^{-1}$ commutes with the relevant actions.

It also follows from \eqref{equalaction}, as we have explicitly done in previous sections, that $t_f$ induces a (well defined) map $T_f:H^1_{\d}(\G,N_\mu)\to H^{1}_{\d}(\G_\mu,N)$. Clearly, since $f:N_\mu\to N$ is a group homomorphism, the map $t_f$ maps the trivial cocycle of $Z^{1}_{\d}(\G,N_\mu)$ to that of $Z^{1}_{\d}(\G_\mu,N)$. Thus, $T_{f}$ maps distinguished element to distinguished element. Furthermore, one readily checks that the induced map $T_{f^{-1}}:H^1_{\d}(\G_\mu,N)\to H^{1}_{\d}(\G,N_\mu)$ is the inverse of $T_f$, and hence $T_f$ is indeed an isomorphism of pointed sets.
\end{proof}

Here is the promised finiteness result.

\begin{corollary}\label{finitecohoapply}
If $H^1_{\d}(\G,G/N)$ is finite and $H^1_{\d}(\G, N_\mu)$ is also finite, for all $\mu\in H^1_{\d}(\G,G)$, then $H^1_{\d}(\G,G)$ is finite as well.
\end{corollary}
\begin{proof}
Since $H^1_{\d}(\G, G/N)$ is finite, there are $\mu_1,\dots,\mu_r\in H^1_{\d}(\G, G)$ such that 
$$H^1_{\d}(\G, G/N)= \{\pi^1(\mu_1),\dots, \pi^1(\mu_r)\}.$$
Note that the $\pi^1$-fibre in $H^1_{\d}(\G, G)$ over $\pi^1(\mu_i)$ coincides with the (pointed) set $\P(\mu_i)$ as defined above. By Lemma \ref{surjection} and Theorem~\ref{twoiso}, this set is a homomorphic image of the cohomology $H^1_{\d}(\G, N_{\mu_i})$ where $N_{\mu_i}$ is the $A$-form of $N$ associated to $\mu_i$. By our assumption, the latter cohomology set is finite, and so the $\pi^1$-fibre in $H^1_{\d}(\G, G)$ over $\pi^1(\mu_i)$ is finite for all $i=1,\dots,r$. It follows that $H^1_{\d}(\G, G)$ is finite.
\end{proof}

\begin{remark}\label{finalremark} \
\begin{enumerate}
\item [(i)] In the case when $G$ is abelian the conclusion of the above corollary holds under the weaker assumption that $H^1_{\d}(\G,G/N)$ and $H^1_{\d}(\G,N)$ are finite. This follows immediately from Theorem~\ref{cohoseq} as in this case the cohomology sequence is an exact sequence of groups (and group homomorphisms).
\item [(ii)] In the context of classical (algebraic) Galois cohomology (see Example~\ref{examples}(i)), the corollary is well known and appears in \S6.4 of \cite{PR} and also in \S4 of Chapter III of \cite{Serrebook}. Our arguments are based on (but are somewhat simpler and more explicit than) those appearing in \cite{PR} (cf. \S1.3, \S2.2, and \S6.4 of this reference).
\item [(iii)] In the context of Kolchin's (differential) constrained cohomology (see Example~\ref{examples}(ii)), the above corollary (which is not treated in Kolchin's book~\cite{Kolchinbook2} or elsewhere, to the author's knowledge) is used in \cite[Lemma 2.6]{LS-Pillay}. There, a sketch of the argument is pointed out (for the differential case) leaving details to the reader.
\end{enumerate}
\end{remark}

\section{Differential Galois cohomology}\label{DGC}
In this section we complete the results on differential constrained cohomology of differential algebraic groups over bounded, differentially large, fields from \cite{LS-Pillay}. We also clarify a proof from \cite{Kamensky-Pillay} about the algebraic Galois cohomology of algebraic groups over bounded fields.  

For model-theoretic notions, differential algebraic notions and their inter-relations, see  the introductions to \cite{LS-Pillay} and \cite{Pillay-PV}, and the various references there (such as Poizat's \cite{Poizat-groupes-stables}). But we recall some key definitions:  Throughout fields are assumed to be of characteristic zero. A field $K$ is said to be {\em bounded} or equivalently has Serre's property (F), if $K$ has only finitely many extensions of degree $n$, for each $n\in \mathbb N$.  A field $K$ is said to be {\em large} if whenever $V$ is a $K$-irreducible variety with a smooth $K$-point, then $V$ has a Zariski-dense set of $K$-points. A differential field $(K,\partial)$ is said to be {\em differentially large} if $K$ is large as a field, and for any differential field extension $(L,\partial)$, if $K$ is existentially closed in $L$ (as fields), then $(K,\partial)$ is existentially closed in $(L,\partial)$. 
We will also be using notation from Example 1.3. 

\medskip

We will make use of a slight generalization or  version of Corollary~\ref{finitecohoapply}.   So here we are in the same general context as the previous section: $M$ is a structure, $A$ a subset of the universe of $M$ as in Assumption 1.1, $G$ is a group definable in $M$ over $A$, $N$ a normal subgroup of $G$ definable in $M$ over $A$, $\pi$ denotes the surjective homomorphism $G\to G/N$, and $\G = Aut(M/A)$. 

\begin{corollary} Let $\kappa$ be any infinite cardinal. Suppose that $|H^{1}_{\d}(\G, G/N)| \leq \kappa$ and $|H^1_{\d}(\G,N_{\mu})| \leq \kappa$ for all $\mu\in H^{1}_{\d}(\G, G)$.  Then $|H^{1}_{\d}(\G, G)| \leq \kappa$.
\end{corollary}
\begin{proof} This has an identical proof as Corollary~\ref{finitecohoapply}: choose $\mu_{i}$ for $i<\lambda\leq \kappa$ such that  $H^1_{\d}(\G, G/N) = \{\pi^1(\mu_{i}): i< \lambda\}$.  For each $i$, the $\pi^1$-fibre over $\pi^1(\mu_{i})$ is the image of $H^1_{\d}(\G, N_{\mu_{i}})$ under a suitable map, so has cardinality at most $\kappa$, hence $H^1_{\d}(\G, G)$ has cardinality at most $\kappa$. 
\end{proof}

The following was proved in \cite[Theorem 5.2]{Kamensky-Pillay}, but some of the subtleties around exact sequences of cohomology sets were overlooked. We take the opportunity
to give the proof again.

\begin{proposition} Suppose the field $K$ is bounded and $G$ is any algebraic group over $K$. Then $H^{1}(K,G)$ is countable.   

\end{proposition}
\begin{proof} The first point is to reduce to the case when $G$ is connected. Let $G^0$ be the connected component of $G$. The exact sequence $G^{0} \to G \to G/G^{0}$  (of algebraic groups over $K$) induces the exact sequence $H^{1}(K, G_{0}) \to H^{1}(K,G) \to H^{1}(K, G/G^{0})$. 
But the latter is finite  by \cite[Proposition 8, Chapter III]{Serrebook}.  Hence by Corollary 5.1, $H^{1}(K,G)$ is countable if and only if $H^{1}(K,G^{0})$ is countable. So we may assume that $G = G^{0}$ is connected.

We then have the exact sequence $1\to L\to G\to A \to 1$ of connected algebraic groups over $K$ where $L$ is linear and $A$ is an abelian variety.  As correctly shown in the proof in \cite[Theorem 5.2]{Kamensky-Pillay}, $H^{1}(K, A)$ is countable (as it coincides with $H^{1}(K_{0},A)$ where $K_{0}$ is a countable elementary substructure of $K$). On the other hand, it is proved in \cite{Serrebook} that $H^{1}(K, L)$ is finite. So, by Corollary 5.1, $H^{1}(K,G)$ is countable. 
\end{proof}

We now pass to differential algebraic groups over differential fields. As in \cite{LS-Pillay} we work with one derivation (so the ambient theory is DCF$_{0}$), although everything generalizes suitably to the case of a finite set $\Delta$ of commuting derivations  (and the corresponding theory DCF$_{0,m}$). In \cite[Theorem 4.1]{LS-Pillay}  we proved finiteness of the differential constrained cohomology $H^{1}_{\partial}(K,G)$ when $G$ is a linear differential algebraic group over a bounded field $K$ with $(K,\partial)$ differentially large. 

\medskip

Here we aim towards proving:

\begin{theorem}  Suppose that $K$ is  bounded (as a field) and $(K,\partial)$ is differentially large. Then, for any differential algebraic group $G$ over $K$, $H^{1}_{\partial}(K,G)$ is countable. 
\end{theorem}

We will go through a few lemmas. For us, $G$ will denote a differential algebraic group over a differential field $(K,\partial)$; equivalently, a group definable over $K$  in a differentially closed field $(\mathcal U,\partial)$ containing $K$.  There is no harm in taking $\mathcal U$ to be a differential closure $K^{diff}$ of $K$.

\begin{lemma}  Suppose $G$ is finite and $K$ is bounded (as a field). Then $H^{1}_{\partial}(K,G)$ is finite. 
\end{lemma} 
\begin{proof}  We could consider $G$ as an algebraic group and use Lemma 2.6(1) from \cite{LS-Pillay}  which says that for $G$ definable over $K$ in the field language, $H^{1}(K,G) = 
H^{1}_{\partial}(K,G)$ and the latter is finite by Proposition 8 of \cite{Serrebook}. 

Alternatively, note that the finitely many points  of $G$ are in $K^{alg}$ so really $H^{1}_{\partial}(K,G) = H^{1}(Gal(K^{alg}/K), G)$ which is finite by Proposition 8 of \cite{Serrebook} again. 
\end{proof}

By the above lemma, together with Corollary 5.1, we may assume that $G$ is {\em connected} as a differential algebraic group, namely has no proper definable subgroup of finite index. 

\begin{lemma} Suppose that $G$ has finite Morley rank, $K$ is bounded as a field, and $(K,\partial)$ is differentially large. Then, $H^{1}_{\partial}(K,G)$ is countable. 
\end{lemma}
\begin{proof}  This is an adaptation of Case 1 of the proof of Theorem 4.1 of \cite{LS-Pillay}. We give some details.  First (cf. Remark 2.4(1)  of \cite{LS-Pillay}), $G$ is definably over $K$ isomorphic to  the ``sharp points" or ``$\partial$-points" of a connected algebraic $\partial$-group $(H,s)$ over $K$.  By a (connected) algebraic $\partial$-group $(H,s)$ over $K$ we mean a connected algebraic group $H$ over $K$ equipped with an extension of $\partial$ to a derivation of the structure sheaf of $H$; equivalently, a regular homomorphic section $s$ (over $K$) of  the surjective homomorphism   $\tau(G) \to G$ of algebraic groups over $K$, $\tau$ being the first prolongation of $G$, which is a connected algebraic group over $K$ as well as a torsor for the tangent bundle $T(G)$ of $G$.  The group $(H,s)^{\partial}$ is  the definable (in the differential field $K^{diff}$) over $K$ group $\{a\in H(K): \partial(a) = s(a)\}$. 
So we are assuming that $G = (H,s)^{\partial}$. 
Any definable over $K$ principal homogeneous space ($PHS$) for $G$ is of the form $(X,s_{X})^{\partial}$ for some (algebraic) principal homogeneous space $X$ over $K$  for $H$, and  
regular section $s_{X}$ over $K$ for $\tau(X)\to X$ such that for all $h\in H$ and $x\in X$, $s_{X}(h\cdot x) =  s(h)\cdot s_{X}(x)$.  (See Remark 2.2 and 2.3 of \cite{LS-Pillay}). One of 
the main points of \cite[\S3]{LS-Pillay} is Corollary 3.3 there which says that: given also $(Y, s_{Y})$, if $X$ and $Y$ are isomorphic over $K$ as algebraic $PHS$'s for $H$, then $(X,s_{X})^{\partial}$ and $
(Y,s_{Y})^{\partial}$ are isomorphic over $K$ as definable (differential algebraic) $PHS$'s for  $G = (H,s)^{\partial}$.  

It follows that the cardinality of $H^{1}_{\partial}(K,G)$ is at most the cardinality of  $H^{1}(K,H)$ which by Proposition 5.2 is countable as $K$ is bounded.
\end{proof}

Before completing the proof of Theorem 5.3   let us recall the ``Manin maps"  and their properties (with precise references):
\begin{fact} Let $A$ be an abelian variety over $K$. Then there is a definable (over $K$) surjective homomorphism $\mu$ from $A({\mathcal U})$ to $({\mathcal U},+)^{d}$, where $d$ is the dimension of $A$ as an algebraic group, such that $\ker(\mu)$  has finite Morley rank (and is connected). 
\end{fact}
\begin{proof} $A$ is an almost direct product of simple abelian varieties $A_{1},\dots,A_{m}$ over $K$, where simple means having no proper infinite algebraic subgroups.  For each $i$, there is by Fact 1.7 (ii) of \cite{DGTIII} a definable (over $K$) surjective homomorphism $\mu_{i}$ from $A_{i}({\mathcal U})$ to $({\mathcal U}^{d_{i}}, +)$ with $\ker(\mu_i)$ connected of finite Morley rank, and where $d_{i} = dim(A_{i})$.  Then the $\mu_{i}$'s induce  the required $\mu: A({\mathcal U}) \to  ({\mathcal U}^{d}, +)$. (Namely we have $\oplus_{i}\mu_{i}: \oplus_{i}A_{i} \to {\mathcal U}^{d}$, which is $0$ on the torsion elements, so induces $\mu:A\to {\mathcal U}^{d}$.) 
\end{proof}

We finish with the proof of Theorem 5.3.

\begin{proof}[Proof of Theorem 5.3]
Let $G$ be our connected definable (over $K$) group. By Corollary 4.2 of \cite{Pillay-groups}, we may assume that $G$ is a subgroup of 
$G_{1}({\mathcal U})$ for some connected algebraic group $G_{1}$ over $K$.  The Chevalley-Barsotti theorem mentioned earlier  gives an exact sequence $1\to L_1 \to G_{1} \to A_1 \to 1$ of
connected algebraic groups over $K$ where $L_{1}$ is linear and $A_{1}$ is an abelian variety.   Let $L = G\cap L_{1}$. So we obtain an exact sequence $1\to L\to G \to A \to 1$ of differential algebraic groups over $K$, where $L = G\cap L_{1}(\mathcal U)$ and  $A\subseteq A_{1}({\mathcal U})$.  Theorem 4.1 of \cite{LS-Pillay} says that $H^{1}_{\partial}(K,L')$ is finite for any $K$-form $L'$ of $L$  (as $L'$ will also be linear). 
So, by Corollary 5.1, it suffices to prove that  $H^{1}_{\partial}(K, A)$ is countable. Let $\mu: A_{1}({\mathcal U}) \to ({\mathcal U}^{d},+)$ be the Manin map given by  Fact~5.6. Then, restricting $\mu$ to $A$  gives an exact sequence of definable over $K$ groups 
$$1 \to  Ker(\mu)\cap A  \to A  \to  N \to 1$$  
where  $Ker(\mu)\cap A$ has finite Morley rank and  $N$ is linear.

By Lemma 5.5 (and 5.4),  $H^{1}_{\partial}(K, Ker(\mu)\cap A))$ is countable, and the same holds also for any  $K$-form of  $Ker(\mu)\cap A$ (as it also has finite Morley rank). 
 By Theorem~4.1 of \cite{LS-Pillay}, $H^{1}_{\partial}(K,N)$ is finite.   Hence, by Corollary 5.1, $H^{1}_{\partial}(K,A)$ is countable.  
 %This completes the proof of Theorem 5.3. 
\end{proof}

\bibliographystyle{plain}

\end{document}